\newif\ifsiam
\newif\ifnummat

 \siamtrue

 \nummatfalse 

\ifsiam
 \documentclass[final,leqno]{siamltex}
\else
  \ifnummat
     \documentclass[nummat]{svjour}
  \else
     \documentclass[twocolumn,fleqn,runningheads,final]{svjour2}
     \journalname{Computing and Visualization in Science}
     \usepackage{mathptmx}
  \fi

\fi

\usepackage{curves,amsmath,amssymb,color,fancybox,graphicx,exscale,enumerate}

\definecolor{dark}{gray}{0.6}
\definecolor{light}{gray}{0.8}

\def\eref#1{(\ref{#1})}

\def\Re{\mathbb R}

\def\mmAuthor{Maximilian S.~Metti}

\ifsiam
    \def\mmShortAuthor{M.~S.~Metti}
\else
    \def\mmShortAuthor{Metti}
\fi

\def\jxAuthor{Jinchao Xu}
\ifsiam
    \def\jxShortAuthor{J.~Xu}
\else
    \def\jxShortAuthor{Xu}
\fi

\def\clAuthor{Chun Liu}
\ifsiam
    \def\clShortAuthor{C.~Liu}
\else
    \def\clShortAuthor{Liu}
\fi

\def\jxAddress{Department of Mathematics, The Pennsylvania State University,
 University Park, Pennsylvania 16802. Email: xu@math.psu.edu.}
 \def\clAddress{Department of Mathematics, The Pennsylvania State University,
 University Park, Pennsylvania 16802. Email: cxl41@psu.edu.}
 \def\mmAddress{Department of Mathematics, The Pennsylvania State University,
 University Park, Pennsylvania 16802. Email: msm37@psu.edu.}

\def\jxThanks{The work of this author was supported by xxxxxxxxxxxxx.}
\def\clThanks{The work of this author was supported by xxxxxxxxxxxxx.}
\def\mmThanks{The work of this author was supported by xxxxxxxxxxxxx.}

\title{Energetically stable discretizations for charge carrier transport and electrokinetic models}

\def\shortTitle{Finite element analysis for PNP}

\def\myKeywords{Finite elements, Poisson-Nernst-Planck, stability analysis, energy estimate}

\def\myAMS{65M15, 65M50, 65M60}

\def\myAbstract{
A finite element discretization using a method of lines approached is proposed for approximately solving the Poisson-Nernst-Planck (PNP) equations.
This discretization scheme enforces positivity of the computed solutions, corresponding to particle density functions,
and a discrete energy estimate is established that resembles the familiar energy law for the PNP system.
This energy estimate is extended to finite element solutions to an electrokinetic model, which couples the PNP system with the Navier-Stokes equations.
Numerical experiments are conducted to validate convergence of the computed solution and verify the discrete energy estimate.
}

\begin{document}


\ifsiam
  \author{\clAuthor%
         \thanks{\clAddress\ \clThanks}
         \and
         \mmAuthor%
         \thanks{\mmAddress\ \mmThanks}
         \and
         \jxAuthor%
         \thanks{\jxAddress\ \jxThanks}
         }
  \maketitle

  \begin{abstract}\myAbstract\end{abstract}
  \pagestyle{myheadings}
  \thispagestyle{plain}
  \markboth{\clShortAuthor,\ \mmShortAuthor,\ and \jxShortAuthor }{\shortTitle}

\else


  \ifnummat
      \author{\clAuthor%
             \thanks{\clThanks}
             \and
             \mmAuthor%
             \thanks{\mmThanks}
             \and
             \jxAuthor%
             \thanks{\jxThanks}
              }
  \else
      \author{\clAuthor%
             \thanks{\clThanks}
             \and
             \mmAuthor%
             \thanks{\mmThanks}
             \and
             \jxAuthor%
             \thanks{\jxThanks}
              }

  \fi
  \institute{\clShortAuthor : \clAddress \\ \mmShortAuthor : \mmAddress \\  \\ \jxShortAuthor : \jxAddress}
  \date{Received: \today\  / Accepted: date}
  \maketitle
  \begin{abstract}\myAbstract\end{abstract}
  \begin{subclass}\myAMS \end{subclass}
\fi




\section{Introduction and Background}
Charge transport refers to any physical process where charged particles interact through an electric field and are driven by an electromagnetic force in some way.
These systems have been observed throughout the history of science; 
they are, naturally, the foundation of electric engineering; and they are commonplace in everyday devices and physical systems, 
from mobile phones to solar-cell batteries, and weather to biology.
The relevant quantities and their relationships are often modeled mathematically by the Maxwell equations, which were first published in \cite{maxwell}.
To this day, phenomena relating to the transport and interaction of charged particles provides a broad variety of research topics in the physical sciences, mathematics, 
and engineering.

In our study, we focus on electrostatic systems, where magnetic forces are absent.
Such systems arise in biological settings, studying semiconductor devices, and electrokinetic systems, where charged particles interact with charged fluids,
to name a few examples.
Our model for charge transport is described by the Poisson-Nernst-Planck (PNP) system of differential equations.
The electric field is defined by Gauss' law in Maxwell's equations,
and the flux of the charged particles are driven by processes of diffusion and drift from the electrostatic force,
which traces back to Nernst \cite{nernst} and Planck \cite{planck}.
This model is valid for systems where charge carriers can be accurately modeled as point charges.

These equations serve as the basis for modeling many devices, such as batteries \cite{batterydft14,batteryxu02}, semiconductor devices 
\cite{BRF83,BBFS89,Jerome85,Jerome96,mock,slotboom73},
fluidic micro/nano-channels and mixers \cite{ericksonli02,karnikdaiguji07,li2004electrokinetics,shin2005mixing}, 
and biological ion channels \cite{liumajdxu14,eisenberglinliu12,luholst10,lueisenberg13}, to name a few.
As a system of coupled nonlinear partial differential equations, the PNP equations lead to a rich source of problems for pde analysis,
where the system and its modifications are studied to improve understanding of the existence, uniqueness, and stability of a solution 
\cite{BD00,BHN94,Jerome85,ryham09}.

Due to the wide variety of devices modeled by the PNP equations (or some modification thereof), 
computer simulation for this system of differential equations is a major application.
This has led to a great deal of literature focusing on numerical solvers for the PNP systems 
\cite{BRF83,BBFS89,solver,Jerome96,luholst10,batterydft14,PS09,PS10,saccostynes98,lueisenberg13,batteryxu02,liumajdxu14}.
Providing a comprehensive numerical analysis would require an energy estimate to establish the stability of the discretization,
some notion of convergence of the computed solution to the true solution, and well-posedness of the discrete problem.
Prohl and Schmuck carried out such an analysis for the PNP system \cite{PS09} and the PNP system coupled with the Navier-Stokes equation \cite{PS10}
using a numerical scheme that employs the method of lines, a finite elements discretization, and fixed point iteration. 
In this work, a novel finite element discretization is used that uses a logarithmic transformation of the charge carrier densities,
which yields several favorable properties, such as automatic positivity of the solution densities and energetic stability of the numerical solution
for both the drift-diffusion model and the electrokinetic model.

In \S 2, we define the PNP equations, introduce the energy law corresponding to the PNP system, propose our discretization 
and prove an energy estimates fully discrete solutions of the PNP system.
In \S 3, we provide a similar analysis for the PNP system coupled with an incompressible fluid,
where the divergence-free property of the fluid plays an essential role in establishing stability;
consequently, this section also contains a discussion of a discontinuous Galerkin approximation for the Navier-Stokes system,
which is known to preserve this divergence-free property.
In \S 4, some numerical experiments are carried out to validate the convergence properties of the numerical solver and the discrete energy estimate.
Some closing remarks are given in \S 5.

\def\jump#1{[\![ #1 ]\!]}

\section{The Poisson-Nernst-Planck equations and its discretization}

The PNP system models the interaction of $N\ge2$ ionic species through an electrostatic field.
We denote the ion density of the $i^\mathrm{th}$ species by $\rho_i>0$ and the electrostatic potential by $\phi$.
Let $\Omega \subset \Re^d$ for $d=1, 2,$ or $3$, and $T$ be a positive and finite real number.
Then, the PNP system is described by the initial value problem:
\begin{align}
	-\nabla\cdot(\epsilon\nabla \phi)	&=	 e_c\sum_{i=1}^N q_i \rho_i,	\label{poisson}	\\
	\frac{\partial \rho_i}{\partial t}	&=	-\nabla \cdot \vec{J}_i,				& i=1,\ldots,N,	\label{mass cons}\\
				\vec{J}_i	&= - D_i \nabla\rho_i - q_i \mu_i \rho_i \nabla \phi,	& i=1,\ldots,N,	\label{NP}
\end{align}
in $\Omega\times(0,T]$, where
\begin{align*}
								\rho_i(x,0)	&=	\rho_{i,0}(x),	&\mathrm{for}\ x\in \Omega,\ i=1,\ldots,N,
\intertext{and}
	-\nabla\cdot\big(\epsilon\nabla \phi(x,0)\big)	&=	 e_c\sum_{i=1}^N q_i \rho_{i,0}(x),	&\mathrm{for}\ x\in \Omega.
\end{align*}
The electric permittivity, $\epsilon = \epsilon_r\epsilon_0>0$, 
represents the the vacuum permittivity constant, $\epsilon_0$, and the material dependent relative permittivity, $\epsilon_r$,
which may be discontinuous in general.
The electric permittivity measures the strength of the long-range (nonlocal) interactions of charged particles.
The term, $e_c$, represents the elementary charge constant.
The ionic flux for the $i^\mathrm{th}$ ion species is denoted by $\vec{J}_i$ and
is defined in \eref{NP} by a model proposed by Nernst \cite{nernst} and Planck \cite{planck},
where $D_i>0$ is the diffusivity, $q_i$ is the valence number, and $\mu_i$ is the mobility of the $i^{\mathrm{th}}$ ion species.
This model is reasonable when the charge carriers are sufficiently small (with respect to the length scale of the domain)
to be accurately modeled as point charges.

We assume the Einstein relation holds, so that we may write
\[
	\mu_i =   e_c D_i /  \kappa_B T^\circ,
\]
where this relation implies that equilibrium distribution of the charge carriers should follow a Boltzmann distribution.
Here, $\kappa_B$ is the Boltzmann constant and $T^\circ$ is the temperature, 
which is considered to be fixed for the purposes of this paper.
For simplicity, we choose our units of measurement such that $e_c = 1$ and $\kappa_B T^\circ = 1$, so that $\mu_i = D_i$.

\subsection{Boundary conditions}
The boundary conditions are a critical component of the PNP model and determine important qualitative behavior of the solution.
A detailed account of stability and existence for steady-state continuous and finite element solutions has reported in \cite{Jerome85,Jerome96}.
For the time-dependent case, existence and stability for the continuous case has been established \cite{BD00,BHN94};
this work is concerned with establishing the stability of finite element solutions for the time dependent PNP equations,
so homogeneous no-flux conditions are considered for each ion species,
\begin{equation} \label{no-ion-flux}
	D_i (\nabla \rho_i + q_i \rho_i \nabla \phi) \cdot \vec n	=	\vec{J}_i\cdot \vec n	= 0,	\qquad\mathrm{on}\ \partial \Omega.
\end{equation}
For the Poisson equation, write a disjoint partition of the boundary: $\partial \Omega = \Gamma_{D} \cup \Gamma_{N} \cup \Gamma_{R}$ with
\begin{align}
									\phi 	&= \delta V	&\mathrm{on}\ \Gamma_{D},	\nonumber\\
				\epsilon \nabla \phi \cdot\vec n	&= S			&\mathrm{on}\ \Gamma_{N},	\label{phi bcs}\\
	\epsilon \nabla \phi \cdot\vec n + \kappa \phi	&= C			&\mathrm{on}\ \Gamma_{R}, \nonumber
\end{align}
where $\delta V$, $S$ and $C$ are given functions.
The Dirichlet boundary condition models an applied voltage, the Neumann condition models surface charges,
and the Robin condition models a capacitor, where $\kappa \ge \bar\kappa > 0$.
Without loss of generality, it is assumed in the case where $\delta V$ is constant, so that $\delta V\equiv 0$.
The capacitance is required to be positive on $\Gamma_R$,  $\kappa > 0$,
though one may take $\Gamma_R = \emptyset$ if no capacitor is to be modeled.
Any combination of Dirichlet, Neumann, and Robin boundary conditions can be applied to $\phi$ for the purposes of this paper,
though the case of pure Neumann boundary conditions requires the an additional constraint,
which can be taken to be $\int_\Omega \phi(x,t)\, dx = 0$ for $0\le t\le T$ so that $\phi$ is uniquely defined.

\subsection{Computational difficulties of the PNP system}
The PNP equations present several difficulties when computing approximate solutions.
Firstly, it is a strongly coupled system of nonlinear equations, which requires an iterative linearization procedure to resolve the nonlinearities,
such as a Newton-Raphson method or fixed point iteration.
While fixed point iteration serves as a helpful tool in the analysis of the PNP system, 
it is difficult to establish its rate of convergence, which is critical in the practice of computing solutions.
Secondly, the Nernst-Planck fluxes given in \eref{NP} are often convection-dominated, 
which leads to several analytical and numerical difficulties, such as the positivity of the ion concentrations, $\rho_i > 0$, 
and the numerical stability of a given discretization.
There are several ways to overcome these issues: the first is to introduce some sort of upwinding scheme,
such as 
the Scharfetter-Gummel scheme and the box method \cite{BCC98,BRF83,schargummel},
or the edge-averaged finite element (EAFE) method \cite{LZ12,XZ99}.
Another option is to introduce a nonlinear change of variables,
such as the Slotboom variables \cite{batterydft14,slotboom73,lueisenberg13} or the quasi-Fermi variables \cite{BCC98,BRF83,Jerome85,Jerome96},
which symmetrize the Nernst-Planck flux.

In this work, a novel change of variables converts the convection-dominated Nernst-Planck flux into a nonlinear-diffusion flux,
similar to the quasi-Fermi variables. 
As a matter of fact, this change of variables is directly related to the quasi-Fermi variables, 
though the quasi-Fermi variables introduce a nonlinear coupling between the equations in the time derivative term in \eref{mass cons}.

\subsection{Energy of the PNP system}
For PNP systems satisfying no-flux boundary conditions on the ion concentrations \eref{no-ion-flux}, 
it is known that the ion concentrations satisfy the conservation of mass,
\[
	\int_{\Omega} \rho_i(x,t)\, dx	=	\int_{\Omega} \rho_{i,0}(x)\, dx,	\qquad\mathrm{for}\ 0 \le t \le T.
\]
Furthermore, in the presence of homogenous Dirichlet boundary conditions on $\phi$,
the stability of the solution to the PNP system is known \cite{BD00,BHN94} to be given by the energy law
\begin{multline}
	\frac{d}{dt} \left\{ \int_\Omega   \sum_{i=1}^N \rho_i (\log \rho_i -1) + \frac{\epsilon}{2}\left| \nabla\phi \right|^2\, dx
			+ \int_{\Gamma_R} \frac{\kappa}{2} \left| \phi \right|^2\, ds \right\}	\\
		=	- \int_\Omega  \sum_{i=1}^N{D_i} \rho_i \left| \nabla \left( \log\rho_i +  q_i \phi \right)\right|^2\ dx,  \label{energy law}
\end{multline}
where the functional,
\[
	\int_\Omega   \sum_{i=1}^N \rho_i (\log \rho_i -1) + \frac{\epsilon}{2}\left| \nabla\phi \right|^2\, dx
			+ \int_{\Gamma_R} \frac{\kappa}{2} \left| \phi \right|^2\, ds,
\]
is referred to as the \emph{energy} and
\[
	\int_\Omega  \sum_{i=1}^N{D_i} \rho_i \left| \nabla \left( \log\rho_i +  q_i \phi \right)\right|^2\ dx \ge 0,
\]
as the \emph{rate of dissipation}.
The physical relevance of the no-flux boundary conditions on the ion concentration and the no-voltage boundary conditions on $\phi$
stem from the notion that the PNP system is energetically closed; that is, there is no direct input or output of energy at the boundary.
However, the case of applying a Dirichlet boundary condition to $\phi$ is critically important in the analysis of many electrostatic devices,
such as semiconductors, protein nano-channels, and electrokinetic devices.
Accordingly, one can show that systems with a Dirichlet boundary condition on $\phi$ still satisfy a similar energy law.

The energy associated with this system takes an unusual form compared to those typically encountered in finite element analysis 
due to the presence of the logarithm;
nevertheless, this identity establishes the stability of the system as well as prescribes its rate of energy dissipation.
Our choice of variables is motivated by the energy law \eref{energy law}, 
which specifies the regularity of the solution: take
\[
	\phi \in \mathcal{H}^1_{\Gamma_D} \equiv 
	\left\{ v \in \mathcal{H}^1(\Omega) \, \Big| \, v|_{\Gamma_{D}} = \delta V \right\},
\]
which is the subpace of the usual $\mathcal{H}^1(\Omega)$ Sobolev space,
and for the ion concentrations,
\[
	\rho_i \in \widetilde{W} \equiv \bigg\{ \rho: \Omega\mapsto \Re \,\Big|\, 
		\int_{\Omega} \rho(\log\rho-1)\, dx < \infty\ \mathrm{and}\ \int_\Omega |\nabla \log\rho|^2\, dx < \infty \bigg\},
\]
leading implicitly to a positivity condition for the ions concentrations.
A log-transformation of the ion concentrations yields a more familiar space
\[
	\eta = \log\rho	\in W \equiv \mathcal{H}^1(\Omega) \cap \mathcal{L}^\infty(\Omega),
\]
and, furthermore, guarantees positivity of the ion concentrations, since $\rho = e^\eta > 0.$

\subsection{Log-density formulation and its energy}

The standard $\mathcal{L}_2(\Omega)$ inner-product is used
\begin{align*}
	( u, v ) &=	\int_{\Omega} uv\,dx,
\intertext{and inner-products on the boundary are given by}
	\langle u, v \rangle_{R}	= \int_{\Gamma_R} uv \,ds,
\end{align*}
and $\langle u, v\rangle_N$ is similarly defined on $\Gamma_N$.

Using the log-density variables, the PNP equations are written in their weak form: find $\eta_i(t) \in W$ with $\eta_{i,t}(t) \in \mathcal{L}^2(\Omega)$
and $\phi \in \mathcal{H}^1_{\Gamma_D}$ such that
\begin{align*}
	(\epsilon\nabla\phi, \nabla v) + \langle \kappa \phi, v \rangle_{\Gamma_{R}}	- \sum_{i=1}^N q_i (e^{\eta_{i}},v)	
							&= \langle C, v \rangle_{\Gamma_{R}} + \langle  S, v \rangle_{\Gamma_{N}},\\
	\Big(\frac{_{\partial}}{^{\partial t}}e^{\eta_{i}}, w \Big)	+	\left( D_i e^{\eta_{i}}\nabla \big( \eta_{i} +  q_i\phi \big),\nabla w \right)	
							&= 0,
\end{align*}
for $i=1,\ldots, N$ and all $v \in V$, $w \in W$, and all times $0< t \le T$, where
\begin{align*}
				\big(\eta_i(\cdot,0), w \big)		&=	\big( \eta_{i,0}, w \big),	&\mathrm{for\ all}\ w \in W,\\
	(\epsilon\nabla\phi(\cdot,0), \nabla v) + \langle \kappa \phi(\cdot,0), v \rangle_{\Gamma_{R}}	
							&= \sum_{i=1}^N q_i (e^{\eta_{i}(\cdot,0)},v) 
							+ \langle C, v \rangle_{\Gamma_{R}} + \langle  S, v \rangle_{\Gamma_{N}},	&\mathrm{for\ all}\ v \in V.
\end{align*}
The energy law written in these new variables takes the form
\begin{multline}
	\frac{d}{dt} \left\{ \int_\Omega   \sum_{i=1}^N e^{\eta_i} (\eta_i -1) + \frac{\epsilon}{2}\left| \nabla\phi \right|^2\, dx
			+ \int_{\Gamma_R} \frac{\kappa}{2} \left| \phi \right|^2\, ds \right\}	\\
		=	- \int_\Omega  \sum_{i=1}^N{D_i} e^{\eta_i} \left| \nabla \left( \eta_i +  q_i \phi \right)\right|^2\ dx.  \label{transformed energy law}
\end{multline}

\subsection{The discrete formulation}
Let $\mathcal{T}_h$ be a triangulation or tetrahedralization of the domain.
For the usual space of piecewise linear polynomials,
\[
	W_h \equiv \left\{ w_h \in \mathcal{H}^1(\Omega) \, \Big| \, v|_{\tau} \in \mathbb{P}^1 \ \mathrm{for \ all} \ \tau \in \mathcal{T}_h \right\} \subset \mathcal{H}^1(\Omega),
\]
and denote the nodal interpolation operator, $\mathcal{I}_h:\mathcal{H}^1(\Omega) \rightarrow W_h$.
When Dirichlet boundary conditions are imposed on the electrostatic potential, 
define the spaces of continuous piecewise linear finite element functions 
\begin{align*}
	V_{h,\Gamma_D} &\equiv \left\{ v_h \in W_h \, \Big| \, v_h|_{\Gamma_{D}} = \mathcal{I}_h(\delta V) \right\},\\
	V_{h,0} &\equiv \left\{ v_h \in W_h \, \Big| \, v_h|_{\Gamma_{D}} = 0 \right\}.
\end{align*}
When Robin boundary conditions are imposed, the lumped boundary inner-product,
\[
	\langle u, v \rangle_{R,h}	= \int_{\Gamma_R} \mathcal{I}_h(uv)\,ds.
\]
is needed to for theoretical purposes to preserve monotonicity of the discrete Poisson equation.
For the time discretization, define a partition of the time domain,
\[
	0 = t_0 < t_1 < \cdots < t_m = T,
\]
where $\Delta t_j \equiv t_j-t_{j-1}$.

The finite element solution to the PNP equations is defined using the above finite element spaces and an implicit time discretization defined on the time partition:
find $\eta_{i,h}^{(j)} \in W_h$ and $\phi_h^{(j)} \in V_{h,\Gamma_D}$ satisfying
\begin{align}
	\left(\epsilon\nabla\phi_h^{(j)}, \nabla v_h\right) + \big\langle \kappa\phi_h^{(j)}, v_h \big\rangle_{R,h}
							- \sum_{i=1}^N q_i \left(e^{\eta_{i,h}^{(j)}},v_h\right)	
							&= \langle C, v_h \rangle_{R,h} + \langle  S, v_h \rangle_{\Gamma_{N}},		\label{fe-poisson} \\
	\frac{1}{\Delta t_j}\big(e^{\eta_{i,h}^{(j)}}, w_h \big)	+ \left( D_i e^{\eta_{i,h}^{(j)}}\nabla \big( \eta_{i,h}^{(j)} +  q_i\phi_h^{(j)} \big),\nabla w_h \right)	
							&= \frac{1}{\Delta t_j}\big(e^{\eta_{i,h}^{(j)}}, w_h \big),	\label{fe-nernst-planck}
\end{align}
for $i=1,\ldots, N$ and all $v_h \in V_{h,0}$, $w_h \in W_h$, and $j = 1,\ldots, m$.
The initial condition is given by
\begin{align}
				\big(\eta_{i,h}^{(0)}, w_h \big)	&=	\big( \eta_{i,0}, w_h \big),	&\mathrm{for\ all}\ w_h \in W_h, &\quad i=1,\ldots,N,\\
	\big(\epsilon\nabla\phi_h^{(0)},\nabla v_h \big)	+ \big\langle \kappa\phi_h^{(0)}, v_h \big\rangle_{R,h}
					 &=	\sum_{i=1}^N q_i\big( e^{\eta_{i,0}}, v_h \big),	\nonumber\\
					 &\quad + \langle C, v_h \rangle_{R,h} + \langle  S, v_h \rangle_{\Gamma_{N}} &\mathrm{for\ all}\ v_h \in V_{h,0}.	\label{fe-init-phi}
\end{align}

\subsection{A discrete maximum principle}
The presence of a nonzero Dirichlet boundary condition imposes additional constraints on the finite element mesh  
in order to maintain a discrete maximum principle for $\phi_h$.
Two approaches for satisfying a discrete maximum principle are summarized in the following lemma.
The first approach constrains the interior angles of the mesh so that the discrete differential operator is monotone,
the second approach requires quasi-uniformity and a sufficiently refined mesh.

Consider an element, $\tau\in\mathcal{T}_h$.
The term \emph{facet} is used below to denote an element edge when $d=2$, and an element face when $d=3$.
Let $E$ be an edge (1-dimensional sub-simplex) of $\tau$.
The $d-2$ dimensional simplex in $\tau$ that is opposite to the edge, $E$, is denoted by $k_E^\tau$.
(In two-dimensions, $|k_E^\tau|=1$.)
The angle, $\theta_E^\tau$, is the angle between the facets containing edge $E$.
The average value of $\epsilon$ on element $\tau$ is given by $\langle \epsilon \rangle_\tau = \int_\tau \epsilon\,dx/ |\tau|$.
In \cite{XZ99}, it was shown that the off-diagonal entries of the stiffness matrix corresponding to the vertices on edge $E$ are given by,
\[
	\omega_E \equiv \frac{1}{d(d-1)} \sum_{\tau \supset E} \langle\epsilon\rangle_\tau|k_E^\tau| \cot \theta_E^\tau \ge 0,
\]
where the summation $\sum_{\tau\supset E}$ is taken to be the summation over all elements $\tau\in \mathcal{T}_h$ containing edge $E$.
In the case where $\epsilon$ is constant, this condition simply requires $\mathcal{T}_h$ to be a Delaunay mesh.
Using this identity, a necessary and sufficient condition is given for the Poisson matrix to be monotone, 
implying that it has a nonnegative inverse and, consequently, a discrete maximum principle.

\begin{lemma} \label{discrete weak maximum principle}
Suppose that one of the following assumptions hold:
\begin{enumerate}[i.]
	\item On each edge $E \in \mathcal{T}_h$, it holds that $\omega_E \equiv \frac{1}{d(d-1)} \sum_{\tau \supset E} \langle\epsilon\rangle_\tau|k_E^\tau| \cot \theta_E^\tau \ge 0,$
	\item The permittivity, $\epsilon$, is constant and $\mathcal{T}_h$ is quasi-uniform and sufficiently refined.
\end{enumerate}
Then, the finite element approximation, $\phi_{h,D} \in V_{h,\Gamma_D}$ defined by
\[
	(\epsilon\nabla\phi_{h,D}, \nabla v_h) + \langle \phi_{h,D}, v_h \rangle_{R,h}
							= 0	\qquad\mathrm{for\ all}\ v_h \in V_{h,0},
\]
satisfies a weak maximum principle
\begin{equation*}
	\max_{x\in\Omega} |\phi_{h,D}(x)| \le C_{\infty} \max_{x\in\Gamma_D} | \mathcal{I}_h(\delta V)|,
\end{equation*}
for some $C_\infty \ge 1$ that only depends on $\epsilon$, $\kappa$, and the shape of $\Omega$.
Furthermore, under condition (i), the bounding constant is given, $C_\infty=1$.
\end{lemma}

The proof of case ($i$) follows from the monotonicity of the discrete differential operator and the proof for case ($ii$) can be referenced from \cite{schatz80}.
In these works, there is no imposed Robin boundary condition, though the mass lumping discretization of this term preserves the discrete maximum principle.

\subsection{A discrete energy estimate}
For autonomous boundary conditions, $\delta V, S,$ and $C$, the stability of the finite element solution analogous to \eref{energy law} is verified.
\begin{theorem} \label{discrete electrostatic stability}
Suppose $\eta_{i,h}^{(j)} \in W_h$ and $\phi_h^{(j)} \in V_{h,\Gamma_D}$ satisfy equations \eref{fe-poisson}---\eref{fe-init-phi} for $i=1,\ldots, m$
and that one of the assumptions in Lemma \ref{discrete weak maximum principle} is satisfied.
Then, the mass is conserved for each ion species,
\begin{equation}	\label{fe-mass-cons}
	\int_{\Omega} e^{\eta_{i,h}^{(j)}(x,t)}\,dx	=	\int_{\Omega} e^{\eta_{i,0}(x)}\,dx,		\quad\mathrm{for}\ i=1,\ldots,N,\quad j=1,\ldots,m,
\end{equation}
and the energy estimate is satisfied,
\begin{multline}
	\max_{1\le j \le m} \left\{ \int_\Omega   \sum_{i=1}^N e^{\eta_{i,h}^{(j)}} (\eta_{i,h}^{(j)} -1) + \frac{\epsilon}{2}\left| \nabla\phi_h^{(j)} \right|^2\, dx
			+ \frac{1}{2}\int_{\Gamma_R}  \mathcal{I}_h\big(\kappa\big(\phi_h^{(j)}\big)^2\big)\, ds \right\}	\\
			+\sum_{j=1}^m \Delta t_j \int_\Omega  \sum_{i=1}^N{D_i} e^{\eta_{i,h}^{(j)}} \left| \nabla \left( \eta_{i,h}^{(j)} +  q_i \phi_h^{(j)} \right)\right|^2\ dx\\
	\le  \int_\Omega   \sum_{i=1}^N e^{\eta_{i,h}^{(0)}} (\eta_{i,h}^{(0)} -1) + \frac{\epsilon}{2}\left| \nabla\phi_h^{(0)} \right|^2\, dx
			+ \frac{1}{2} \int_{\Gamma_R}  \mathcal{I}_h\big(\kappa\big(\phi_h^{(0)}\big)^2\big)\, ds + C_1,
\end{multline}
where $C_1$ depends on the number of ion species, their initial masses, the electric permittivity coefficient, and $C_\infty$.
In the cases of no Dirichlet boundary conditions or a homogeneous Dirichlet boundary condition on $\phi_h$, the constant $C_1$ vanishes.
\end{theorem}

\begin{proof}
To prove the conservation of mass for the ion species, choose $w_h \equiv 1 \in W_h$ in equation \eref{fe-nernst-planck} to show
\[
	\frac{1}{\Delta t_j} \int_\Omega e^{\eta_{i,h}^{(j)}}-e^{\eta_{i,h}^{(j-1)}}\, dx	
	= \bigg(\frac{e^{\eta_{i,h}^{(j)}}-e^{\eta_{i,h}^{(j-1)}}}{\Delta t_j}, 1 \bigg) + \left( D_i e^{\eta_{i,h}^{(j)}}\nabla \big( \eta_{i,h}^{(j)} +  q_i\phi_h^{(j)} \big),\nabla 1 \right)	= 0,	
\]
which yields \eref{fe-mass-cons}.
This argument expectedly fails when Dirichlet boundary conditions are imposed on $\eta_{i,h}$, since $1 \not\in W_h$.
This is not an artifact of the discretization, however, and is also the case for the continuous system.

For the energy estimate, set $w_h = \eta_{i,h}^{(j)} + q_i \phi_h^{(j)} \in W_h$, which is a valid choice for the test function since $\phi_h^{(j)} \in V_{h,\Gamma_D} \subseteq W_h$.
This gives
\[
	\bigg(\frac{e^{\eta_{i,h}^{(j)}}-e^{\eta_{i,h}^{(j-1)}}}{\Delta t_j}, \eta_{i,h}^{(j)} \bigg) 
		+ q_i \bigg(\frac{e^{\eta_{i,h}^{(j)}}-e^{\eta_{i,h}^{(j-1)}}}{\Delta t_j}, \phi_h^{(j)} \bigg)
		=	-\int_{\Omega} D_i e^{\eta_{i,h}^{(j)}} \left|\nabla\left(\eta_{i,h}^{(j)} + q_i \phi_h^{(j)} \right) \right|^2\, dx,
\]
which is summed over $i=1,\ldots, N$, to get
\begin{multline}	\label{fe-bound-i}
	\sum_{i=1}^N\bigg(\frac{e^{\eta_{i,h}^{(j)}}-e^{\eta_{i,h}^{(j-1)}}}{\Delta t_j}, \eta_{i,h}^{(j)} \bigg) 
		+ \sum_{i=1}^N q_i \bigg(\frac{e^{\eta_{i,h}^{(j)}}-e^{\eta_{i,h}^{(j-1)}}}{\Delta t_j}, \phi_h^{(j)} \bigg)	\\
		=	-\sum_{i=1}^N \int_{\Omega} D_i e^{\eta_{i,h}^{(j)}} \left|\nabla\left(\eta_{i,h}^{(j)} + q_i \phi_h^{(j)} \right) \right|^2\, dx.
\end{multline}
The first terms on the left are bounded by using the convexity of the function $f(\rho) = \rho(\log\rho-1)$ for $\rho>0$, which can be used to show
\[
	(\rho_j - \rho_{j-1}) \log\rho_j \ge \rho_j(\log\rho_j-1) - \rho_{j-1}(\log\rho_{j-1}-1).
\]
This follows from $f'(\rho) = \log\rho$, $f''(\rho) = 1/\rho > 0$, and Taylor expansion.
Applying this bound with $\rho_j = e^{\eta_{i,h}^{(j)}}$ and $\rho_{j-1} = e^{\eta_{i,h}^{(j-1)}}$, one obtains for each $i$
\begin{equation} \label{fe-bound-ii}
	\bigg(\frac{e^{\eta_{i,h}^{(j)}}-e^{\eta_{i,h}^{(j-1)}}}{\Delta t_j}, \eta_{i,h}^{(j)} \bigg) 
		\ge \frac{\big(e^{\eta_{i,h}^{(j)}}, \eta_{i,h}^{(j)} - 1\big) - \big(e^{\eta_{i,h}^{(j-1)}}, \eta_{i,h}^{(j-1)} - 1\big)}{\Delta t_j}.
\end{equation}

To bound the remaining term on the left side of \eref{fe-bound-i}, decompose $\phi_h^{(j)} = \phi_{h,0}^{(j)} + \phi_{h,D}$,
where $\phi_{h,0}^{(j)} \in V_{h,0}$ and $\phi_{h,D} \in V_{h,\Gamma_D}$ satisfies the steady differential equation subject to the interpolated Dirichlet boundary condition:
\begin{equation} \label{fe-dirichlet}
	(\epsilon\nabla\phi_{h,D}, \nabla v_h) + \langle \kappa\phi_{h,D}, v_h \rangle_{R,h}
							= 0,	\qquad \phi_{h,D}|_{\Gamma_D} = \mathcal{I}_h(\delta V),
\end{equation}
for all $v_h \in V_{h,0}$.
Write
\begin{equation}\label{fe-bound-iii}
	\sum_{i=1}^N q_i \bigg(\frac{e^{\eta_{i,h}^{(j)}}-e^{\eta_{i,h}^{(j-1)}}}{\Delta t_j}, \phi_h^{(j)} \bigg)
		= \sum_{i=1}^N q_i \bigg(\frac{e^{\eta_{i,h}^{(j)}}-e^{\eta_{i,h}^{(j-1)}}}{\Delta t_j}, \phi_{h,0}^{(j)} \bigg)
		   + \sum_{i=1}^N q_i \bigg(\frac{e^{\eta_{i,h}^{(j)}}-e^{\eta_{i,h}^{(j-1)}}}{\Delta t_j}, \phi_{h,D} \bigg)
\end{equation}
and bound the first term on the right by subtracting consecutive time-steps of \eref{fe-poisson} and taking $v_h = \phi_{h,0}^{(j)} \in V_{h,0}$,
\begin{align}
	\sum_{i=1}^N q_i \bigg(\frac{e^{\eta_{i,h}^{(j)}}-e^{\eta_{i,h}^{(j-1)}}}{\Delta t_j}, \phi_{h,0}^{(j)} \bigg)
		&= \left(\epsilon \frac{\nabla\phi_{h,0}^{(j)}-\nabla\phi_{h,0}^{(j-1)}}{\Delta t_j},\nabla\phi_{h,0}^{(j)}\right)
			+ \left\langle\kappa \frac{\phi_{h,0}^{(j)}-\phi_{h,0}^{(j-1)}}{\Delta t_j},\phi_{h,0}^{(j)}\right\rangle_{R,h} \nonumber\\
		&= \left(\epsilon \frac{\nabla\phi_{h}^{(j)}-\nabla\phi_{h}^{(j-1)}}{\Delta t_j},\nabla\phi_{h}^{(j)}\right)
			+ \left\langle\kappa \frac{\phi_{h}^{(j)}-\phi_{h}^{(j-1)}}{\Delta t_j},\phi_{h}^{(j)}\right\rangle_{R,h} \nonumber\\
		&\ge \frac{\big(\epsilon \nabla\phi_{h}^{(j)}, \nabla\phi_{h}^{(j)} \big) - \big(\epsilon\nabla\phi_{h}^{(j-1)},\nabla\phi_{h}^{(j-1)}\big)}{2\Delta t_j} \label{fe-bound-iv}\\
		&\qquad\qquad\qquad+ \frac{\big\langle \kappa\phi_{h}^{(j)}, \phi_{h}^{(j)}\big\rangle_{R,h} 
			- \big\langle \kappa\phi_{h}^{(j-1)}, \phi_{h}^{(j-1)}\big\rangle_{R,h}}{2\Delta t_j},	\nonumber
\end{align}
where the second equality follows from adding and subtracting the term
$$\Delta t_i^{-1}\big[(\epsilon \nabla\phi_{h,D},\nabla\phi_{h,0}^{(j)}) + \langle\kappa \phi_{h,D},\phi_{h,0}^{(j)}\rangle_{R,h}\big]$$
and the definition of $\phi_{h,D}$, \eref{fe-dirichlet}.

Combining \eref{fe-bound-i}---\eref{fe-bound-ii} and \eref{fe-bound-iii}---\eref{fe-bound-iv} gives the bound
\begin{multline}	\label{fe-telescope}
	\frac{ \left[\mathcal{E}_{h}^{(j)} + \sum_{i=1}^N q_i \big(e^{\eta_{i,h}^{(j)}}, \phi_{h,D} \big) \right]
		- \left[\mathcal{E}_{h}^{(j-1)} + \sum_{i=1}^N q_i \big(e^{\eta_{i,h}^{(j-1)}}, \phi_{h,D} \big) \right] }{ \Delta t_j}\\
	\le -\sum_{i=1}^N \int_{\Omega} D_i e^{\eta_{i,h}^{(j)}} \left|\nabla\left(\eta_{i,h}^{(j)} + q_i \phi_h^{(j)} \right) \right|^2\, dx,
\end{multline}
where $\mathcal{E}_h^{(k)}$ denotes the discrete energy functional,
\[
	\mathcal{E}_{h}^{(k)} \equiv \int_\Omega   \sum_{i=1}^N e^{\eta_{i,h}^{(k)}} (\eta_{i,h}^{(k)} -1) + \frac{\epsilon}{2}\left| \nabla\phi_h^{(k)} \right|^2\, dx
									+ \frac{1}{2}\int_{\Gamma_R} \mathcal{I}_h\big(\kappa\big(\phi_h^{(k)}\big)^2\big)\, ds.
\]
The first term in \eref{fe-telescope} yields a telescoping sum; a Gr\"onwall argument leads to
\begin{multline*}
	\max_{1\le j \le m} \left[\mathcal{E}_{h}^{(j)} + \sum_{i=1}^N q_i \big(e^{\eta_{i,h}^{(j)}}, \phi_{h,D} \big) \right]
		+ \sum_{j=1}^m \Delta t_j \sum_{i=1}^N \int_{\Omega} D_i e^{\eta_{i,h}^{(j)}} \left|\nabla\left(\eta_{i,h}^{(j)} + q_i \phi_h^{(j)} \right) \right|^2\, dx \\
	\le \mathcal{E}_{h}^{(0)} + \sum_{i=1}^N q_i \big(e^{\eta_{i,h}^{(0)}}, \phi_{h,D} \big).
\end{multline*}

To complete the proof, the conservation of mass bounds
\begin{multline}\label{voltage bound}
	\sum_{i=1}^N q_i \big(e^{\eta_{i,h}^{(0)}}-e^{\eta_{i,h}^{(k)}}, \phi_{h,D} \big)
		\le 2 \sum_{i=1}^N \big\| q_i e^{\eta_{i,h}^{(0)}} \big\|_{\mathcal{L}^1(\Omega)} \| \phi_{h,D} \|_{\mathcal{L}^\infty}
		\\\le 2N \big(\max_{1\le i \le N} \big\| q_i e^{\eta_{i,h}^{(0)}} \big\|_{\mathcal{L}^1(\Omega)}\big) \| \phi_{h,D} \|_{\mathcal{L}^\infty},
\end{multline}
where $\big\| q_i e^{\eta_{i,h}^{(0)}} \big\|_{\mathcal{L}^1(\Omega)}$ is directly proportional to the ionic mass, determined by the initial condition.
The estimate, 
\[
	\| \phi_{h,D} \|_{\mathcal{L}^\infty} \le C_\infty\max_{x\in\Gamma_D}|\mathcal{I}_h(\delta V)|,
\] 
follows from  Lemma \ref{discrete weak maximum principle}.
In the case where $\delta V \equiv 0$ or $\Gamma_D = \emptyset$, it is clear that $\phi_{h,D} \equiv 0$ so that this term vanishes altogether.
\end{proof}

To conclude this section, one important remark is in order.
The inequality of this energy estimate is a consequence of only two aspects of the discretization;
first, the time discretization satisfies \eref{fe-bound-ii} and \eref{fe-bound-iv} only with an inequality,
whereas the semi-discrete solution (continuous in time) satisfies these bounds with equality.
The only other inequality in the proof of Theorem \ref{discrete electrostatic stability} is used to bound non-homogeneous Dirichlet boundary conditions.
As a matter of fact, in the semi-discrete case with homogeneous Dirichlet boundary conditions (or no Dirichlet boundary conditions),
the finite element solution satisfies the energy estimate with equality.

\section{Electrokinetics}

Electrokinetic systems combine effects of electrostatic systems coupled with incompressible fluid flow.
The model equations studied here couple the PNP equations with the incompressible Navier-Stokes (NS) equations.
This system of equations models electrokinetic phenomena such as electroosmosis, electrophoresis, streaming potentials, 
electrowetting, and many other phenomena where charged particles and fluids interact \cite{ericksonli02,karnikdaiguji07,li2004electrokinetics,shin2005mixing}.
Some analysis for this system in the continuous case is given in \cite{ryham09}.
The equations governing the electrokinetic system seek a solution, $\eta_1,\ldots,\eta_N, \phi, \vec u$, and $p$, such that
\begin{align}
	-\nabla\cdot(\epsilon\nabla \phi)	&=	 \sum_{i=1}^N q_i e^{\eta_i},		\label{ek-poisson}\\
	\frac{_\partial}{^{\partial t}}e^{\eta_i}	&=	\nabla \cdot \big( D_i e^{\eta_i} \nabla (\eta_i + q_i \phi ) - e^{\eta_i} \vec u \big),	& i=1,\ldots,N,\label{ek-np}\\
	\rho_f \big( \vec u_t + ( \vec u \cdot \nabla ) \vec u \big) + \nabla p &= \nabla \cdot \big( 2\mu \varepsilon( \vec u) \big) - \sum_{i=1}^N q_i e^{\eta_i}\nabla \phi,	
						\label{cont continuity eqn}\\
	\nabla \cdot \vec u	&= 0, \label{cont div free}
\end{align}
on $\Omega\times(0,T]$, where $\varepsilon(\cdot)$ denotes the symmetrized vector gradient,
\[
	\varepsilon(\vec u) = \frac12\big( \nabla \vec u + (\nabla \vec u)^T \big).
\]
The initial conditions for this system are
\begin{align*}
	\eta_i(x,0)	=	\eta_{i,0}(x),	&&	-\nabla\cdot(\epsilon\nabla \phi(x,0))	&=	 \sum_{i=1}^N q_i e^{\eta_{i,0}(x)}, 	\\ 
	\vec u(x,0)	=	\vec u_{0}(x), 	&&	p(x,0) &= p_0(x),		& \mathrm{for}\ x\in \Omega.
\end{align*}
Equations \eref{ek-poisson} and \eref{ek-np} come directly from the PNP model,
where an additional coupling term in \eref{ek-np} models a kinetic force from the fluid flow described by the Navier-Stokes equations.
Equations \eref{cont continuity eqn} and \eref{cont div free} are the usual NS equations for an incompressible fluid.
In equation \eref{cont continuity eqn}, the electrostatic force $\sum_{i=1}^N q_i e^{\eta_i}\nabla\phi$ models the effects of the PNP on the fluid.

The boundary conditions considered for the PNP variables remain the same as the previous section \eref{no-ion-flux}--\eref{phi bcs}. 
The Navier-Stokes boundary conditions are assumed to be some combination of no-flux and no-slip boundary conditions,
\begin{align*}
						\vec u\cdot \vec n 		&= 0,	&\mathrm{on}\ \Gamma_{\mathrm{no-flux}} \subseteq \partial\Omega,\\
	\big(\mu\varepsilon( \vec u) \vec n\big)\cdot \vec t	&= 0,	&\mathrm{on}\ \Gamma_{\mathrm{no-flux}},\\
									\vec u	&= \vec 0,	&\mathrm{on}\ \Gamma_{\mathrm{no-slip}}\equiv \partial\Omega\setminus\Gamma_{\mathrm{no-flux}}.
\end{align*}
Due to the incompressibility condition on the fluid velocity \eref{cont div free}, the solution satisfies $\varepsilon (\vec u ) = \nabla \vec u$,
which is commonly used to represent the viscosity term in the continuity equation \eref{cont continuity eqn}.
This identity does not hold, however, for general $\vec s \in [\mathcal{H}^1(\Omega)]^d$, 
so one must take care when using the divergence theorem to write the PNP-NS system in weak form; namely,
for $\vec s =\vec 0$ on $\Gamma_\mathrm{no-slip}$,
\[
	-\big( \nabla \cdot \big( 2\mu \varepsilon( \vec u) \big), \vec s \big) = \big( 2 \mu \varepsilon(\vec u), \varepsilon(\vec s) \big).
\]
In the special case when $\vec u, \vec s \equiv \vec 0$ on $\partial\Omega$, 
the right side reduces to $\big( \mu \nabla \vec u, \nabla \vec s \big)$.

The corresponding energy law for this system is given by
\begin{multline}
	\frac{d}{dt} \left\{ \int_\Omega   \frac{\rho_f}{2}|\vec{u}|^2 + \sum_{i=1}^N \rho_i (\log \rho_i -1) + \frac{\epsilon}{2}\left| \nabla\phi \right|^2 \, dx
			+ \int_{\Gamma_R} \frac{\kappa}{2} \left| \phi \right|^2\, ds \right\}	\\
		=	- \int_\Omega  \frac\mu2 |\varepsilon(\vec u)|^2 + \sum_{i=1}^N{D_i} \rho_i \left| \nabla \left( \log\rho_i +  q_i \phi \right)\right|^2\ dx.  \label{PNPNS energy law}
\end{multline}
The terms in the energy law relating to the NS variables are critically hinged on specific mathematical structures of the NS system: 
in particular, the divergence-free property of the fluidic velocity plays a significant role in the cancelation of the cross-terms between the PNP and NS systems.
As a result, the discrete solution must satisfy the divergence-free property on every subdomain of $\Omega$.
This can be accomplished in several ways, using higher order elements or locally discontinuous Galerkin (DG) approximations \cite{ABMXZ14,CKS05}, for example.
For many practical applications, solutions using higher order elements may be prohibitively expensive to compute;
the discussion below primarily considers DG approximations for the NS variables.

To define the weak solution to the NS equations, let 
\begin{align*}
	Q &\equiv \mathcal{L}_2(\Omega),	\\
	S &\equiv \{ \vec s \in [\mathcal{H}^1(\Omega)]^d \, |\, \vec u\cdot \vec n =  0\ \mathrm{on}\ \Gamma_{\mathrm{no-flux}},
							\quad  \vec u = \vec 0\ \mathrm{on}\ \Gamma_{\mathrm{no-slip}}	\},\\
	\mathcal{H}(\mathrm{div};\Omega) 		&\equiv \{ \vec s \in [\mathcal{L}_2(\Omega)]^d\, |\, \nabla \cdot \vec s \in \mathcal{L}_2(\Omega) \},	\\
	\mathcal{H}_0(\mathrm{div};\Omega) 	&\equiv \{ \vec s \in \mathcal{H}(\mathrm{div};\Omega) \,|\, \vec s \cdot \vec n = 0\ \mathrm{on}\ \partial \Omega \},	\\
	\mathcal{H}_{D,0}(\mathrm{div};\Omega) 	&\equiv \{ \vec s \in \mathcal{H}_0(\mathrm{div};\Omega) \,|\, \vec s = \vec 0\ \mathrm{on}\ \Gamma_{\mathrm{no-slip}} \},	\\
	\mathcal{H}_0(\mathrm{div}^0;\Omega) 	&\equiv \{ \vec s \in \mathcal{H}_0(\mathrm{div};\Omega) \,|\, \nabla \cdot \vec s \equiv 0\ \mathrm{in}\ \Omega \}.
\end{align*}
The weak solution of the NS equations is $(\vec u, p) \in S\times Q$ satisfying
\begin{align*}
	D_t ( \vec u; \vec u, \vec s ) + A( \vec u, \vec s ) + B( \vec s, p )	&= ( \vec f, \vec s ),	&\mathrm{for\ all}\ \vec s\in S,	\\
										B( \vec u, q )	&= 0,	&\mathrm{for\ all}\ q \in Q,
\end{align*}
where $\vec f \in [\mathcal{L}_2(\Omega)]^d$ and 
\begin{align*}
	D_t ( \vec w; \vec u, \vec s )	&\equiv \rho_f ( \vec u_t, \vec s ) 
				+ \frac{\rho_f}{2} \big( (\vec w\cdot \nabla) \vec u, \vec s \big),	\\
	A(\vec u, \vec s)	&\equiv \big( 2\mu \varepsilon( \vec u), \varepsilon( \vec s) \big),	\\
	B(\vec u, q)		&\equiv -( \nabla \cdot \vec u, q ).
\end{align*}

The well-posedness of the weak formulation can be demonstrated using Babu\v{s}ka-Brezzi theory \cite{brezzifortin91},
where a Korn inequality must be established, as in the following lemma, which comes directly from \cite{ABMXZ14}.

\begin{lemma}\label{korn inequality}
Let $\Omega \subset \Re^d, d=2,3$ be a polygonal or polyhedral domain. Then, there exists a positive constant $C_{K}$
(depending on the domain through its diameter and shape) such that
\begin{equation} \label{korn bound}
	| \vec s |_{1} \le C_{K} \| \varepsilon (\vec s) \|_0,	\qquad\mathrm{for\ all}\ \vec s \in S.
\end{equation}
\end{lemma}

\subsection{The discrete formulation}
Recall that $\mathcal{T}_h$ denotes the finite element mesh on $\Omega$ and let $\mathcal{E}_h$ denote the set of interior element facets.
The broken $\mathcal{L}_2$ and $\mathcal{H}^1$ inner-products and norms are defined in the usual way
\[
	(p, q)_{\mathcal{T}_h} \equiv	\sum_{\tau \in \mathcal{T}_h} ( p, q )_{\tau},	\quad
	\| q \|_{0,\mathcal{T}_h} \equiv	( q, q )_{\mathcal{T}_h}^{1/2},	\quad\mathrm{and}\quad
	| s |_{1,\mathcal{T}_h} \equiv	( \nabla s, \nabla s )_{\mathcal{T}_h}^{1/2},
\]
for $p,q \in \mathcal{L}_2(\Omega)$ and 
$s \in \mathcal{H}^1(\mathcal{T}_h) \equiv \{ v \in \mathcal{L}_2(\Omega)\, |\ \, v|_{\tau} \in \mathcal{H}^1(\tau)\ \mathrm{for\ all}\ \tau \in \mathcal{T}_h \}.$

Let $w \in \mathcal{H}^1(\mathcal{T}_h)$, $\vec s \in [\mathcal{H}^1(\mathcal{T}_h)]^d$, $\sigma \in [\mathcal{H}^1(\mathcal{T}_h)]^{d\times d}$
denote a scalar, vector, and rank-two tensor field, respectively.
These fields are $\mathcal{H}^1$-regular within each element, though inter-element continuity is not assumed.
Fix $e \in \mathcal{E}_h$, where $e = \tau_+ \cap \tau_-$.
Denote the outward unit normal vectors of $\tau_+$ and $\tau_-$ by $\vec n_+$ and $\vec n_-$, respectively;
the averages across $e$ on internal facets are defined by
\[
	\{ w \} \equiv \frac12(w_+ + w_-),	\quad	\{ \vec s \} \equiv \frac12(\vec s_+ + \vec s_-),	\quad\mathrm{and}\quad	\{\sigma\} \equiv \frac12(\sigma_+ + \sigma_-),
\]
and given by their traces on the boundary facets;  the jumps across internal facets are given by
\begin{align*}
	[\![ w ]\!] &\equiv w_+ \vec n_+ + w_-\vec n_-,	&[ \vec s ] &\equiv \vec s_+ \cdot \vec n_+ + \vec s_- \cdot \vec n_-,\\
	[\![ \vec s ]\!] &\equiv \vec s_+ \otimes \vec n_+ + \vec s_- \otimes \vec n_-,
	&[ \sigma ] &\equiv \sigma_+ \vec n_+ + \sigma_- \vec n_-,
\end{align*}
and $[\![ w ]\!] = w \vec n,\	[ \vec s ] = \vec s \cdot \vec n,\ [\![ \vec s ]\!] = \vec s \otimes \vec n,$ and $[ \sigma ] = \sigma \vec n$ on boundary facets.
where the subscripts on the functions are equipped with their natural meanings of restriction to the element $\tau_+$ or $\tau_-$.
An inner-product defined over the inter-element facets is defined
\[
	\langle w, v \rangle_{\mathcal{E}_h}	\equiv \sum_{e\in\mathcal{E}_h}\int_e w(s) v(s) \,ds.
\]
Using the facet average and jump notation, the following identities are readily verified by direct computation: for $\vec s \cdot \vec n=0$ on $\partial \Omega$,
\[
	\sum_{\tau\in\mathcal{T}_h} \int_{\partial\tau} ( \vec s \cdot \vec n_\tau ) w\, ds	
		=	\langle [\![ w ]\!], \{\vec s\} \rangle_{\mathcal{E}_h}+ \langle [\vec s],\{w\}\rangle_{\mathcal{E}_h},
\]
and
\[
	\sum_{\tau\in\mathcal{T}_h} \int_{\partial\tau} ( \sigma \vec n_\tau ) \cdot \vec s\, ds	=	
				\langle [\![ \vec s ]\!] , \{\sigma\}\rangle_{\mathcal{E}_h} + \langle[ \sigma ],\{\vec s\}\rangle_{\mathcal{E}_h}.
\]

To preserve the local divergence-free property of the fluid velocity, 
nonconforming finite elements are useful for assigning degrees of freedom aimed at preserving this property instead of conforming to the continuous spaces.
We require the finite element space for the pressure $Q_h \subset Q$ and $S_h \subset [\mathcal{H}^1(\mathcal{T}_h)]^d \cap \mathcal{H}_{D,0}(\mathrm{div})$,
where $S_h \not\subset S$, in general.
While it is not necessary that $S_h$ conforms to $S$, several constraints are imposed on the finite element pair $S_h\times Q_h$ 
to ensure well-posedness of the discrete problem.
First, it is required that
\begin{equation}
	\nabla \cdot S_h \subseteq Q_h,	\label{div conforming}
\end{equation}
and, second, that there exists for each $q_h \in Q_h$ a corresponding $\vec s_h \in S_h$ such that 
\begin{equation}	\label{poincare}
	\nabla \cdot s_h = q_h	\quad\mathrm{and}\quad	\| s_h \|_0 \le c_P \| q_h \|_0,
\end{equation}
where $c_P>0$ is a Poincar\'e constant that depends on $\Omega$ in general, but not on $q_h$.
Requirements \eref{div conforming} and \eref{poincare} together imply that $\nabla \cdot S_h = Q_h$.
The final requirement for well-posedness is the existence of a local interpolation operator, $\Pi_\tau: S|_{\tau} \rightarrow S_h|_{\tau}$, 
for each element $\tau \in \mathcal{T}_h$ such that
\begin{equation} \label{interpolation}
	| \Pi_\tau \vec s |_{1,\tau} \le C_{0} | \vec s |_{1,\tau}		\quad\mathrm{and}\quad
	\| ( I - \Pi_\tau ) \vec s \|_{0,\tau} \le C_{1} h_\tau | \vec s |_{1,\tau},
\end{equation}
with $h_\tau = \mathrm{diam}(\tau)$.
This local interpolation property is extended over $\mathcal{T}_h$, to yield an interpolation operator, $\Pi_{S_h}: S\rightarrow S_h$.

Some well-studied finite element pairs satisfying \eref{div conforming}--\eref{interpolation} are the \emph{Raviart-Thomas} elements,
\emph{Brezzi-Douglas-Marini} elements, and the \emph{Brezzi-Douglas-Fortin-Marini} elements all of degree $k\ge1$.
Furthermore, as all of these elements are div-conforming, they have continuous normal components across inter-element facets, 
which, loosely speaking, ``reduces'' the discontinuity of the finite element space, requiring simpler penalty functions in the discontinuous formulation.
This additional continuity also plays a role in the cancellation of the PNP-NS cross terms and is commented upon in the proof of Theorem \ref{discrete electrokinetic stability}.

The discrete formulation of the NS equations given by: find $( \vec u_h^{(j)}, p_h^{(j)} ) \in S_h\times Q_h$ such that
\begin{align}	\label{discrete continuity}
	D_{h,t} \big( \vec u_h^{(j)}; \vec u_h^{(j)}, \vec s_h \big) + A_h\big( \vec u_h^{(j)}, \vec s_h \big) + B_h\big( \vec s_h, p_h^{(j)} \big)
								&= \frac{\rho_f}{\Delta t_j} \big( \vec u_h^{(j-1)}, \vec s_h \big) + \big( \vec f(t_j), \vec s_h \big),	&\mathrm{for\ all}\ \vec s_h\in S_h,	\\
	B_h \big( \vec u_h^{(j)}, q_h \big)	&= 0,	&\mathrm{for\ all}\ q_h \in Q_h,	\label{discrete incompressibility}
\end{align}
for $j=1,\ldots, m$, where initial conditions are defined by projection of the initial values, 
$\vec u_h^{(0)} = \Pi_{S_h} \vec u_0$ and $( p_h^{(0)}, q_h ) = ( p_0, q_h )$ for all $q_h \in Q_h$.

The forms used to define the discrete solution are given by
\begin{align*}
	D_{h,t} \big( \vec w_h; \vec u_h, \vec s_h \big)	&\equiv \frac{\rho_f}{\Delta t_j} ( \vec u_h, \vec s_h ) 
				- \frac{\rho_f}{2} \big( (\vec w_h\cdot \nabla) \vec s_h, \vec u_h \big) 
				+ \sum_{\tau\in\mathcal{T}_h} \int_{\partial \tau} (\vec w_h\cdot \vec n_\tau) (\vec{u}_h^w\cdot \vec s_h)\,ds,	\\
	A_h(\vec u_h, \vec s_h)	&\equiv \big( 2\mu \varepsilon ( \vec u_h ), \varepsilon ( \vec s_h ) \big)_{\mathcal{T}_h} 
							-  \langle 2\mu \{ \varepsilon (\vec u_h) \},[\![\vec s_h]\!]\rangle_{\mathcal{E}_h} 
										- \langle 2\mu[\![ \vec u_h ]\!], \{\varepsilon (\vec s_h)\} \rangle_{\mathcal{E}_h}\\
						&\qquad + \alpha \sum_{e\in\mathcal{E}_h} h_e^{-1} \int_e \mu[\![ \vec u_h ]\!]:[\![ \vec s_h ]\!]\,ds,\\
	B_h(\vec u_h, q_h)	&\equiv -( \nabla \cdot \vec u_h, q_h )_{\mathcal{T}_h}.
\end{align*}
These forms are quite standard in the DG literature, though some terms and important properties remain to be specified.

The discrete kinematic derivative term, $D_{h,t}: S_h\times S_h\times S_h\rightarrow \Re$, is defined using the upwind flux, $\vec{u}_h^w$, given by
\[
	\vec{u}_h^w = \lim_{\delta\rightarrow0^+}\vec u_h\big( x-\delta \vec w_h(x) \big).
\]
This definition yields coercivity, summarized by the standard identity \cite{CKS05}:
\begin{equation} \label{kinetic coercivity}
	D_{h,t}(w_h;  \vec u_h, \vec u_h)	= \frac12\sum_{e\in \mathcal{E}_h} \int_e | \vec w_h\cdot \vec n | \big| [\![ \vec u_h ]\!] \big|^2 \,ds,
\end{equation}
where $\vec n$ denotes either unit normal vector to the facet $e$.

The bilinear forms, $A_h$ and $B_h$, are a standard description for a DG discretization of Stokes' equations 
and are motivated by the definition of the weak derivative followed by applying the divergence theorem element-wise.
The parameter, $\alpha > 0$, penalizes discontinuities of the solution across element interfaces and must be chosen to be sufficiently large.

Since the finite element space, $S_h$, is div-conforming \eref{div conforming}, 
equation \eref{discrete incompressibility} implies that $\nabla \cdot \vec u_h = 0$ on each element $\tau \in \mathcal{T}_h$.
Another useful property inherited from \eref{div conforming} is that all $\vec s_h \in S_h$ have continuous normal components across element edges; 
namely, letting $\vec n$ and $\vec t$ denote the normal and tangent unit vectors, respectively, on each edge, $e \in \mathcal{T}_h$, gives
\[
	\vec s_h(x) = (\vec s_h \cdot \vec n ) \vec n + (\vec s_h \cdot \vec t\,) \vec t	= \vec s_h^{\,n}(x)  + \vec s_h^{\,t}(x),
\]
and $[\![ \vec s_h^{\,n} ]\!] = 0$ on each edge.
As a result, it holds that
\[
	 \int_e [\![ \vec s_h ]\!] : \sigma \, ds = \int_e [\![ \vec s_h^{\,t} ]\!] : \sigma \, ds \quad\mathrm{for\ any}\ \sigma \in [\mathcal{H}^1(\Omega)]^{d\times d}.
\]
Using this result, the coercivity of the kinematic derivative term, $D_{h,t}$, reduces to
\[
	D_{h,t} \big( \vec w_h; \vec u_h, \vec u_h \big)	= \frac12\sum_{e\in \mathcal{E}_h} \int_e | \vec w_h\cdot \vec n | \big| [\![ \vec u_h^{\,t} ]\!] \big|^2 \,ds
\]
and, for $A_h$,
\begin{multline*}
	A_h(\vec u_h, \vec s_h)	\equiv \big( 2\mu \varepsilon ( \vec u_h), \varepsilon ( \vec s_h) \big)_{\mathcal{T}_h} 
							-  \langle 2\mu \{ \varepsilon (\vec u_h) \},[\![\vec s_h^{\,t}]\!]\rangle_{\mathcal{E}_h} 
							- \langle2 \mu[\![ \vec u_h^{\,t} ]\!], \{\varepsilon (\vec s_h)\} \rangle_{\mathcal{E}_h}\\
						 	+ \alpha \sum_{e\in\mathcal{E}_h} h_e^{-1} \int_e \mu[\![ \vec u_h^{\,t} ]\!]:[\![ \vec s_h^{\,t} ]\!]\,ds,
\end{multline*}
where only the tangent components along the element facets are penalized, and $B_h$ becomes
\[
	B_h (\vec s_h, q_h ) = ( \nabla \cdot \vec s_h, q_h )	\quad\mathrm{for\ all}\ \vec s_h \in S_h\ \mathrm{and}\ q_h\in Q_h.
\]

The energy norm for the discontinuous fluid velocity is defined by
\[
	\| \vec s \|_{DG}^2 = | \vec s |_{1,\mathcal{T}_h}^2 + | \vec s |_*^2,	\quad\mathrm{where}\quad
	| \vec s |_*^2 \equiv	\sum_{e \in \mathcal{E}_h} h_e^{-1} \| [\![ \vec s_t ]\!] \|_{0,e}^2.
\]
For any of the three finite element examples mentioned above,  one can establish the $\|\cdot\|_{DG}$-stability of the bilinear form, $A_h$,
meaning that there exists a positive constant, $\gamma$, such that
\begin{equation} \label{DG stability}
	A_h( \vec s_h, \vec s_h ) \ge \gamma \| \vec s_h \|_{DG}^2,	\quad\mathrm{for\ all}\ \vec s_h \in S_h \ (\mathrm{see}\, \cite{ABMXZ14}).
\end{equation}
This stability, together with an argument using fixed point iteration \cite{CKS05}, is used to verify the existence of a discrete solution for the NS equations
using this DG scheme.

\subsection{The discrete electrokinetic system}
Employing the discretization of the PNP system given in the previous section and the discretization of the NS system above,
the discrete solution to the electrokinetic system is defined by the finite element functions $\eta_{1,h}^{(j)}, \ldots, \eta_{N,h}^{(j)} \in W_h$,
$\phi_h^{(j)} \in V_{h,\Gamma_D}$, and $(\vec u_h^{(j)}, p_h^{(j)}) \in S_h\times Q_h$ satisfying
\begin{align}
	(\epsilon\nabla{\phi}_h^{(j)}, \nabla v_h) + \langle \kappa{\phi}_h^{(j)}, v_h \rangle_{R,h} 
								&=	 \sum_{i=1}^N q_i \big( e^{\eta_{i,h}^{(j)}},v_h \big) 	\label{PNPNS eqn i}\\
								&\quad+ \langle C^{(j)}, v_h \rangle_{R,h} + \langle  S^{(j)}, v_h \rangle_{\Gamma_{N}},\nonumber \\
	\frac{1}{\Delta t_j}\big( e^{\eta_{i,h}^{(j)}}, w_h \big) + \left( D_i e^{\eta_{i,h}^{(j)}}\nabla \big( \eta_{i,h}^{(j)}	+  q_i\phi_h^{(j)} \big),\nabla w_h \right) 
								&= \frac{1}{\Delta t_j}\big( e^{\eta_{i,h}^{(j-1)}}, w_h \big) 
									+ \left( e^{\eta_{i,h}^{(j)}} \vec u_h^{(j)}, \nabla w_h \right),\label{PNPNS mass cons}\\
	D_{h,t} \big( \vec u_h^{(j)}; \vec u_h^{(j)}, \vec s_h \big) + A_h\big( \vec u_h^{(j)}, \vec s_h \big) + B_h\big( \vec s_h, p_h^{(j)} \big)
								&= \frac{\rho_f}{\Delta t_j} \big( \vec u_h^{(j-1)}, \vec s_h \big) 
									- \sum_{i=1}^N q_i\big( e^{\eta_{i,h}^{(j)}}\nabla \phi_h^{(j)}, \vec s_h \big), \label{ns continuity}\\
	B_h \big( \vec u_h^{(j)}, q_h \big)	&= 0,	\label{pnpns discrete incompressibility}
\end{align}
for all $w_h \in W_h, v_h \in V_{h,0}, \vec s_h \in S_h, q_h \in Q_h$, and $j=1,\ldots, m$.
Initial conditions are prescribed by
\begin{align}
				\big(\eta_{i,h}^{(0)}, w_h \big)	&=	\big( \eta_{i,0}, w_h \big),	&\mathrm{for\ all}\ w_h \in W_h, \\
	\big(\epsilon\nabla\phi_h^{(0)},\nabla v_h \big)	+ \big\langle \kappa\phi_h^{(0)}, v_h \big\rangle_{R,h}
					 &=	\sum_{i=1}^N q_i\big( e^{\eta_{i,0}}, v_h \big),	\nonumber\\
					 &\quad + \langle C, v_h \rangle_{R,h} + \langle  S, v_h \rangle_{\Gamma_{N}} &\mathrm{for\ all}\ v_h \in V_{h,0},\\
							\vec u_h^{(0)}	&= \Pi_{S_h} \vec u_0,\\
					\big( p_h^{(0)}, q_h \big)	&= \big( p_0, q_h \big),	&\mathrm{for\ all}\ q_h \in Q_h,
\end{align}
where $\Pi_{S_h}: S \rightarrow S_h$ satisfies \eref{interpolation} locally.

The stability of the discrete solution of the electrokinetic system is given by the following theorem.
\begin{theorem} \label{discrete electrokinetic stability}
Suppose $\eta_{i,h}^{(j)} \in W_h, \phi_h^{(j)} \in V_{h,\Gamma_D}, \vec u_h^{(j)} \in S_h, p_h^{(j)}$ satisfy equations \eref{PNPNS eqn i}---\eref{pnpns discrete incompressibility},
where $S_h\times Q_h$ is one of the stable Stokes pairs as described above.
Furthermore, suppose the mesh satisfies one of the assumptions of Lemma \ref{discrete weak maximum principle}.
Then, the mass is conserved for each ion species,
\begin{equation*}
	\int_{\Omega} e^{\eta_{i,h}^{(j)}(x,t)}\,dx	=	\int_{\Omega} e^{\eta_{i,0}(x)}\,dx,		\quad\mathrm{for}\ i=1,\ldots,N,\quad j=1,\ldots,m,
\end{equation*}
and the energy estimate is satisfied,
\begin{multline}
	\max_{1\le j \le m} \left\{ \int_\Omega   \frac{\rho_f}{2}\big|\vec u_h^{(j)}\big|^2 + 
			\sum_{i=1}^N e^{\eta_{i,h}^{(j)}} (\eta_{i,h}^{(j)} -1) + \frac{\epsilon}{2}\left| \nabla\phi_h^{(j)} \right|^2\, dx
			+ \frac{1}{2}\int_{\Gamma_R}  \mathcal{I}_h\big(\kappa\big(\phi_h^{(j)}\big)^2\big)\, ds \right\}\\
			+\sum_{j=1}^m \Delta t_j \left[ \gamma \big\| \vec u_h^{(j)} \big\|_{DG}^2
			+ \int_\Omega  \sum_{i=1}^N{D_i} e^{\eta_{i,h}^{(j)}} \big| \nabla \big( \eta_{i,h}^{(j)} +  q_i \phi_h^{(j)} \big) \big|^2\ dx \right]\\
	\le  \int_\Omega   \frac{\rho_f}{2}\big|\vec u_h^{(0)}\big|^2 + \sum_{i=1}^N e^{\eta_{i,h}^{(0)}} (\eta_{i,h}^{(0)} -1) + \frac{\epsilon}{2}\big| \nabla\phi_h^{(0)} \big|^2\, dx
			+ \frac{1}{2}\int_{\Gamma_R}  \mathcal{I}_h\big(\kappa\big(\phi_h^{(0)}\big)^2\big)\, ds + C_1,
\end{multline}
where $C_1$ depends on the number of ion species, their initial masses, the electric permittivity coefficient, and $C_\infty$.
In the cases of no Dirichlet boundary conditions or a homogeneous Dirichlet boundary condition on $\phi_h$, the constant $C_1$ vanishes.
\end{theorem}

\begin{proof}
The proof of Theorem \ref{discrete electrokinetic stability} closely follows that of Theorem \ref{discrete electrostatic stability},
in addition to \eref{DG stability} for the NS variables.
The only remaining terms are the cross terms between the PNP and NS systems, 
which cancel due to the strong divergence-free property of $\vec u_h^{(j)}$ and the continuity of the normal components across element edges.

The conservation of mass follows from choosing $w_h \equiv 1 \in W_h$ in equation \eref{PNPNS mass cons}, as in Theorem \ref{discrete electrostatic stability}.
Following the argument in the proof of Theorem \ref{discrete electrostatic stability} exactly gives
\begin{multline} \label{fe-telescope PNPNS i}
	\frac{ \left[\mathcal{E}_{h}^{(j)} + \sum_{i=1}^N q_i \big(e^{\eta_{i,h}^{(j)}}, \phi_{h,D} \big) \right]
		- \left[\mathcal{E}_{h}^{(j-1)} + \sum_{i=1}^N q_i \big(e^{\eta_{i,h}^{(j-1)}}, \phi_{h,D} \big) \right] }{ \Delta t_j}\\
	\le -\sum_{i=1}^N \int_{\Omega} D_i e^{\eta_{i,h}^{(j)}} \left|\nabla\left(\eta_{i,h}^{(j)} + q_i \phi_h^{(j)} \right) \right|^2\, dx
		+ \sum_{i=1}^N\left( e^{\eta_{i,h}^{(j)}} \vec u_h^{(j)}, \nabla \big(\eta_{i,h}^{(j)} + q_i \phi_h^{(j)} \big) \right)	\\
	= -\sum_{i=1}^N \int_{\Omega} D_i e^{\eta_{i,h}^{(j)}} \left|\nabla\left(\eta_{i,h}^{(j)} + q_i \phi_h^{(j)} \right) \right|^2\, dx
		+ \sum_{i=1}^N \left[ \left( \vec u_h^{(j)}, \nabla e^{\eta_{i,h}^{(j)}} \right) + q_i \left( e^{\eta_{i,h}^{(j)}}  \nabla \phi_h^{(j)}, \vec u_h^{(j)} \right)\right],
\end{multline}
where the discrete energy is recalled as
\[
	\mathcal{E}_{h}^{(k)} = \int_\Omega   \sum_{i=1}^N e^{\eta_{i,h}^{(k)}} (\eta_{i,h}^{(k)} -1) + \frac{\epsilon}{2}\left| \nabla\phi_h^{(k)} \right|^2\, dx
									+ \frac{1}{2}\int_{\Gamma_R}  \mathcal{I}_h\big(\kappa\big(\phi_h^{(k)}\big)^2\big)\, ds.
\]
Let $\zeta_h^{(j)} = \sum_{i=1}^N q_i e^{\eta_{i,h}^{(j)}} \in \mathcal{H}^1(\Omega)$.
Since $\vec u_h^{(j)}$ is strongly divergence free, has a continuous normal component across inter-element facets,
$\vec u_h^{(j)} \cdot \vec n = 0$ on $\partial \Omega$, and $\zeta_h^{(j)}$ is continuous,
\begin{multline} \label{div cancellation}
	\sum_{i=1}^N \left( \vec u_h^{(j)}, \nabla e^{\eta_{i,h}^{(j)}} \right) = \left( \vec u_h^{(j)}, \nabla \zeta_h^{(j)} \right)\\
			= \big( \nabla \cdot \vec u_h^{(j)}, \zeta_h^{(j)} \big)_{\mathcal{T}_h} 
				+ \sum_{\tau\in\mathcal{T}_h} \int_{\partial\tau} (\vec u_h^{(j)} \cdot \vec n_\tau) \zeta_h^{(j)}\,ds = 0.
\end{multline}
Then, combining \eref{fe-telescope PNPNS i} and \eref{div cancellation} provides the bound
\begin{multline}	\label{fe-telescope PNPNS ii}
	\frac{ \left[\mathcal{E}_{h}^{(j)} + \big({\zeta_{h}^{(j)}}, \phi_{h,D} \big) \right]
		- \left[\mathcal{E}_{h}^{(j-1)} + \big({\zeta_{h}^{(j-1)}}, \phi_{h,D} \big) \right] }{\Delta t_j}\\
	\le -\sum_{i=1}^N \int_{\Omega} D_i e^{\eta_{i,h}^{(j)}} \left|\nabla\left(\eta_{i,h}^{(j)} + q_i \phi_h^{(j)} \right) \right|^2\, dx
		+ \sum_{i=1}^N q_i \left( e^{\eta_{i,h}^{(j)}}  \nabla \phi_h^{(j)}, \vec u_h^{(j)} \right).
\end{multline}

For the NS terms, it follows from \eref{kinetic coercivity}
\begin{align}
	D_{t,h}\big(\vec u_h^{(j)};\vec u_h^{(j)}, \vec u_h^{(j)}) - \frac{\rho_f}{\Delta t_j} \big(\vec u_h^{(j-1)},\vec u_h^{(j)}\big)
		&=	\frac{\rho_f}{\Delta t_j} \big( \vec u_h^{(j)} - \vec u_h^{(j-1)}, \vec u_h^{(j)} \big)\\
		&\qquad +	\frac12\sum_{e\in \mathcal{E}_h} \int_e | \vec u_h^{(j)}\cdot \vec n | \big| \big[\!\!\big[ \big(\vec u_h^{(j)}\big)\big]\!\!\big] \big|^2 \,ds\label{discrete kinetic bound}\\
		&\ge \frac{\rho_f}{2\Delta t_j} \big( \big\|\vec u_h^{(j)}\big\|_0^2 - \big\| \vec u_h^{(j-1)} \big\|_0^2 \big), \nonumber
\end{align}
and, choosing $q_h = p_h^{(j)}$ in \eref{pnpns discrete incompressibility},
\begin{equation} \label{div free estimate}
	B_h\big(\vec u_h^{(j)}, p_h^{(j)}\big) = 0.
\end{equation}
Setting $\vec s_h = \vec u_h^{(j)}$ in \eref{ns continuity} and employing the bounds \eref{DG stability}, \eref{discrete kinetic bound}--\eref{div free estimate} gives
\begin{equation} \label{NS energy bound}
	\frac{\rho_f}{2\Delta t_j} \big( \big\|\vec u_h^{(j)}\big\|_0^2 - \big\| \vec u_h^{(j-1)} \big\|_0^2 \big) + \gamma \big\| \vec u_h^{(j)} \big\|_{DG}^2
		\le -\sum_{i=1}^N q_i \left( e^{\eta_{i,h}^{(j)}}  \nabla \phi_h^{(j)}, \vec u_h^{(j)} \right).
\end{equation}

Adding \eref{fe-telescope PNPNS ii} and \eref{NS energy bound} gives
\begin{multline}	\label{pnpns energy telescope}
	\frac{ \left[\frac{{\rho_f}}{2}\|\vec u_h^{(j)}\|_0^2+\mathcal{E}_{h}^{(j)} + \big({\zeta_{h}^{(j)}}, \phi_{h,D} \big) \right]
		- \left[\frac{{\rho_f}}{2}\|\vec u_h^{(j-1)}\|_0^2+\mathcal{E}_{h}^{(j-1)} + \big({\zeta_{h}^{(j-1)}}, \phi_{h,D} \big) \right] }{\Delta t_j}\\
	\le -\gamma \big\| \vec u_h^{(j)} \big\|_{DG}^2 -\sum_{i=1}^N \int_{\Omega} D_i e^{\eta_{i,h}^{(j)}} \left|\nabla\left(\eta_{i,h}^{(j)} + q_i \phi_h^{(j)} \right) \right|^2\, dx.
\end{multline}
A discrete Gr\"onwall argument and invoking \eref{voltage bound} gives the energy estimate.
\end{proof}

\section{Numerical experiments}

This section presents some numerical experiment that verify the viability, efficiency, and accuracy of computed solutions defined the proposed discretization in the above sections.
According to the discretizations in \S\S 2--3, a system of nonlinear elliptic equations must be solved at each time step.
While there are many approaches to solving such a system, two commonly used techniques for resolving the nonlinear behavior are fixed point iteration 
(often referred to as Gummel iteration in the semiconductor literature) and Newton methods.
For purposes of analysis, fixed point iteration is a very important tool, as convergence can be verified for more general problems;
however, as a practical matter, it is often difficult to establish the rate of convergence to the nonlinear solution for this approach.
This practical difficulty motivates the use of a quasi-Newton method for the experiments presented below, 
where the relative residual approaches zero super-linearly.
The nature of the model equations raise many issues concerning the numerical solver, such as resolving nonlinearities, 
upwinding schemes to preserve numerical stability, and describing the solver for the arising linear systems.
Accordingly, some of the details of the numerical solver are deferred to an upcoming publication \cite{solver}.

It is important to mention that the linearized equations resulting from a Newton-type approach lead to systems of linearized pdes that are potentially \emph{convection-dominated}.
This leads to potential algorithmic difficulties in preserving stability for the computed solution; 
so, some form of upwinding must be implemented to ensure accuracy.
The well-studied edge-averaged finite element (EAFE) method is proven to provide stable numerical solutions that do not display
spurious oscillatory behavior \cite{LZ12,XZ99}.
A point of emphasis here is that the nonlinear solution is stable, as verified by the energy estimates above,
thought the sequence of linearized equations are not necessarily stable.

A solver was implemented in C++ that leverages some existing functionality of the FEniCS 1.3.0 \cite{fenics} software package for
generating systems of linear algebraic equations corresponding to an elliptic pde.
Here, the elliptic pdes are the linearized pdes coming from Newton's method, with an EAFE approximation to improve stability.
Once the systems of algebraic equations are constructed, the Fast Auxilary Space Preconditioners (FASP) software package \cite{fasp}
is used to efficiently solve the resulting systems.

The first experiment presented here is designed to establish the rate of convergence for the PNP discretization at steady state;
since the discrete solution is defined using a method-of-lines approach, 
this experiment also verifies the convergence rate of the presented numerical scheme to the solution of the nonlinear elliptic equation at each time step.
Several PNP systems are solved, where the permittivity coefficient, $\epsilon$, is tested for decreasing values.
Testing the solver for small values of $\epsilon$ is important in many practical problems concerning semiconductors and biological applications,
where this coefficient is on the order of $10^{-4}$ to $10^{-8}$ after the system has been non-dimensionalized.
For this experiment, the equation
\begin{align*}
		-\nabla\cdot(\epsilon\nabla \phi)		&=	 e^{\eta_1} - e^{\eta_2} + f_0,	\\
	\frac{_\partial }{^{\partial t}} e^{\eta_1}	&=	\nabla \cdot \left( e^{\eta_1} \nabla(\eta_1 + \nabla \phi) \right) + f_1,	\\
	\frac{_\partial }{^{\partial t}} e^{\eta_2}	&=	\nabla \cdot \left( e^{\eta_2} \nabla(\eta_2  - \nabla \phi) \right) + f_2,
\end{align*}
is solved on the domain $\Omega = [-1,1]\times[-\frac{_1}{^2},\frac{_1}{^2}]\times[-\frac{_1}{^2},\frac{_1}{^2}]$, 
where $f_0, f_1, f_2$ are chosen so that the solution, for $x=(x_1,x_2,x_3)\in\Omega$, is 
\[
	\eta_1(x) = e^{\frac{\log10}{2}(x_1-1)},	\quad \eta_2(x) = e^{-\frac{\log10}{2}(x_1+1)},	
	\quad\mathrm{and}\quad	\phi(x) = -\frac{2\sinh(x_1)}{e-e^{-1}}.
\]
As this experiment is designed to test the numerical convergence to the nonlinear solution at steady state, 
Dirichlet boundary conditions are imposed at the ends of the domain $x_1 = \pm 1$.
The iteration count for convergence to the nonlinear solutions (determined by reducing the relative residual by a factor of $10^{-10}$) are reported in Table \ref{log density table},
along with the $\mathcal{H}^1$ semi-norm of the error, given by
\[
	\mathcal{H}^1\ \mathrm{semi\!-\!norm} = \left( | e^{\eta_1} - e^{\eta_{1,h}} |_1^2     +    | e^{\eta_2} - e^{\eta_{2,h}} |_1^2 + | \phi - \phi_{h} |_1^2 \right)^{1/2}.
\]

It is clear that the Newton iterates converge in a reasonable number of iterations (fewer than 10 in all cases),
which is encouraging for small values of $\epsilon$.
Additionally, the convergence rate is to be linear for all values of $\epsilon$ in Figure \ref{fig-convergence-rate} with respect to the mesh size.

\begin{table}
\begin{center} \begin{tabular}{| c || c | c || c | c || c | c || c | c | } \hline
	$\epsilon$	&	
	\multicolumn{2}{|c||}{$20\times10\times10$}	&	
	\multicolumn{2}{ c||}{$40\times20\times20$}	&	
	\multicolumn{2}{ c||}{$60\times30\times30$}	&	
	\multicolumn{2}{ c |}{$80\times40\times40$}
	\\\hline\hline
	1		&	7	& $2.65\times10^{-3}$		&	6	& $6.67\times10^{-4}$		&	6	& $2.97\times10^{-4}$		&	6	& $1.67\times10^{-4}$	\\
	$10^{-2}$	&	6	& $3.81\times10^{-3}$		&	6	& $9.80\times10^{-4}$		&	5	& $4.38\times10^{-4}$		&	5	& $2.47\times10^{-4}$	\\
	$10^{-4}$	&	5	& $7.03\times10^{-3}$		&	5	& $2.37\times10^{-3}$		&	5	& $1.20\times10^{-3}$		&	5	& $7.18\times10^{-4}$	\\
	$10^{-8}$	&	5	& $7.25\times10^{-3}$		&	9	& $2.61\times10^{-3}$		&	9	& $1.43\times10^{-3}$		&	9	& $9.34\times10^{-4}$	\\\hline
\end{tabular} \end{center}
\caption{The count of Newton iterates to decrease the initial residual by a factor of $10^{-10}$ and the $\mathcal H^1$ semi-norm of the error.} \label{log density table}
\end{table}

\begin{figure}
\begin{center}
	\includegraphics[scale=.25]{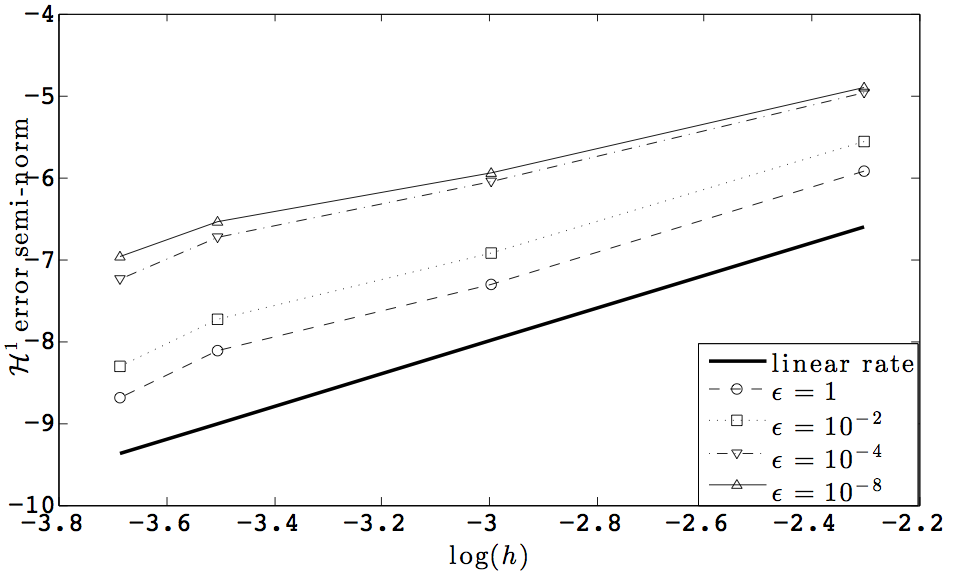}
\caption{The logarithm of the error measured in the $\mathcal{H}^1$ semi-norm, plotted against the logarithm of the element diameter.
The lines depict the log of the error for various values of $\epsilon$, where the thick line is a reference for linear convergence.}\label{fig-convergence-rate}
\end{center}	
\end{figure}

A second experiment validates that the energy estimate is satisfied.
While this property is certainly true for the theoretical finite element solution to the nonlinear problem,
it is important to verify that the numerical solution, computed by inexact iterative methods, preserves this property.
For this experiment, the domain is $\Omega = [-1,1]\times[-\frac{_1}{^{10}},\frac{_1}{^{10}}]\times[-\frac{_1}{^{10}},\frac{_1}{^{10}}]$.
The time domain is $(0,0.03]$, and a uniform time-step of length $\Delta t = \frac{1}{3000}$
For this problem, we solve the system defined by \eref{poisson}---\eref{NP}, with $N=2$, $\mu_1 = \mu_2 = D_1 = D_2 = 1, q_1=1, q_2=-1$, and $\epsilon = \frac{1}{100}$,
with no-flux boundary conditions imposed on the Nersnt-Planck equations and mixed homogeneous Dirichlet and inhomogeneous Neumann boundary conditions on $\phi$:
\begin{align*}
	\phi &= 0,							&&\mathrm{for}\ x_1 = \pm1,\\
	\epsilon\nabla\phi\cdot\vec n	&= 1,	&&\mathrm{for}\ x_1 \le 0\ \mathrm{and}\ x_3 = \frac{1}{10},	&\ \mathrm{or}\ 	x_1 \ge 0\ \mathrm{and}\ x_3 = -\frac{1}{10},\\
	\epsilon\nabla\phi\cdot\vec n	&= -1,	&&\mathrm{for}\ x_1 \le 0\ \mathrm{and}\ x_3 = -\frac{1}{10},&\ \mathrm{or}\ 	x_1 \ge 0\ \mathrm{and}\ x_3 = \frac{1}{10},\\
	\epsilon\nabla\phi\cdot\vec n	&= 0,	&&\mathrm{for}\ x_2  = \pm\frac{1}{10}.
\end{align*}
These boundary conditions model surface charges, of alternating charge, lining opposite sides of a channel along the $x_1$ direction and electrode contacts at the ends of the channel.
The experiment demonstrates for each time step that the discrete energy estimate,
\[
	\delta E^{(j)}  = \frac{\mathcal{E}_h^{(j)} - \mathcal{E}_h^{(j-1)}}{\Delta t} \le -\int e^{\eta_{1,h}} \big|\nabla(\eta_{1,h} + \phi_h) \big|^2 + e^{\eta_{2,h}} \big|\nabla(\eta_{2,h} - \phi_h) \big|^2\, dx 
	\equiv -\Delta^{(j)},
\]
is satisfied until the dissipation is below the tolerance of the nonlinear solver, \mbox{$\Delta^{(j)} < 10^{-8}$.}
To clearly illustrate that the energy estimate is satisfies, Figure \ref{fig-energy} displays the quantity $\log( -\delta E^{(j)}) - \log (\Delta^{(j)})$, 
which is positive when the energy estimate is satisfied.
\begin{figure}
\begin{center}
	\includegraphics[scale=.25]{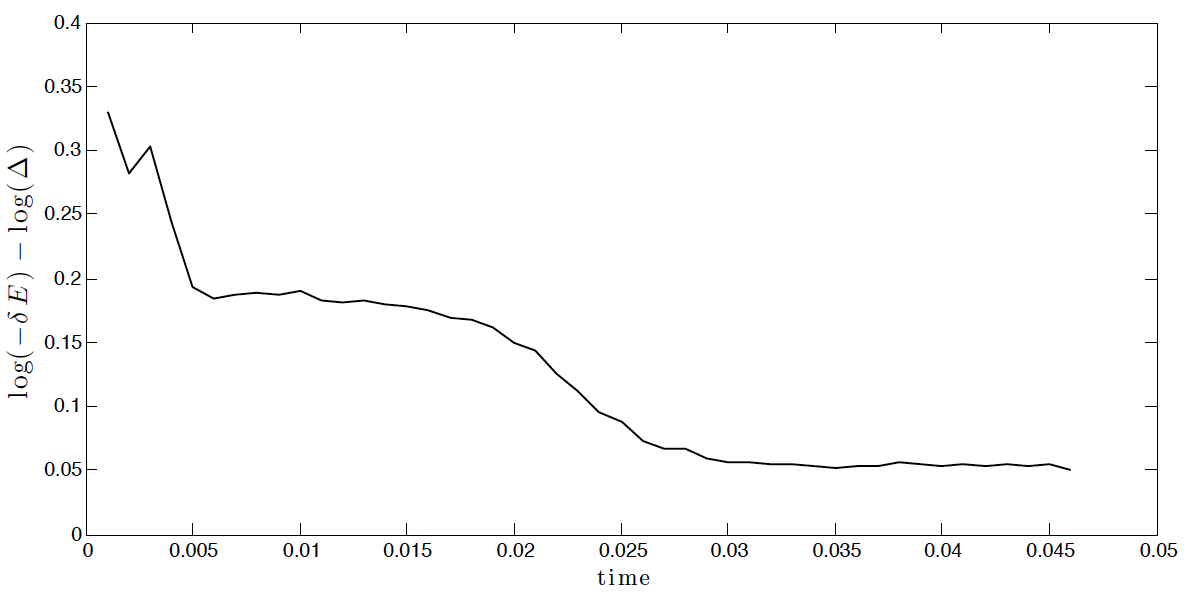}
\caption{The difference, $\log( -\delta E) - \log (\Delta)$, plotted over the time domain until convergence.}\label{fig-energy}
\end{center}	
\end{figure}


\section{Concluding remarks}

In this paper, the energetic stability is established for the finite element solutions to the PNP equations and an electrokinetic model with an incompressible fluid,
with a minor extension to the case of inhomogeneous boundary conditions on the electrostatic potential.
This extension imposes some additional constraints on the finite element mesh so that a weak discrete maximum principle can bound the
energy introduced to the system by this inhomogeneous boundary condition.

This energy estimate for the finite element solutions mimics the energy law of the continuous solution to these models,
where a logarithmic transformation is a key ingredient to establishing the stability for the electrostatic terms and the divergence-free property 
of the discrete solution to the NS equations is essential to the stability of the fluidic variables, as well as the cancellation of cross terms between the two models.
Recall that this divergence-free property of the finite element fluidic velocity is a result of using a DG formulation of the NS equations.

A mathematical justification for the convergence is a matter of future work, 
though the experiments in this paper numerically demonstrate that convergence is obtained for various values of the permittivity coefficient in the Poisson equation.
Furthermore, the computed solution using a quasi-Newton scheme is also shown to satisfy the energy law for the PNP system.
The numerical solver for the PNP equations (with and without coupling to the NS equations) will be described in an upcoming publication \cite{solver}.

\bibliography{ref}
\bibliographystyle{siam}
\end{document}